\documentclass[12pt]{article}
\usepackage{amsfonts,amsmath,amsthm, amssymb}
\usepackage{latexsym, euscript, epic, eepic}
\newtheorem{lemma}{Lemma}[section]
\newtheorem{theorem}{Theorem}[section]
\newtheorem{definition}{Definition}[section]
\newtheorem{proposition}{Proposition}[section]
\newtheorem{corollary}{Corollary}

\newtheorem{remark}{Remark}
\newtheorem{maintheorem}{Theorem}

\newcommand{\g}{\gamma}

\def \d{\delta}
\def \e{\epsilon}
\def \g{\gamma}

\begin{document}

\title{Pressures  for Asymptotically  Sub-additive Potentials Under a Mistake Function \footnotetext {* Corresponding author}
 }
\author{  Wen-Chiao Cheng$^\dag$ and Yun Zhao$^{\ddag ,*}$ and Yongluo Cao$^{\ddag}$\\
\small \it $\dag$ Department of Applied Mathematics\\
\small \it Chinese Culture University\\
\small \it Yangmingshan, Taipei, Taiwan, 11114\\
\small \it e-mail: zwq2@faculty.pccu.edu.tw\\
\small \it $\ddag$ Department of Mathematics\\
\small \it Suzhou University\\
\small \it Suzhou 215006, Jiangsu, P.R.China\\
\small \it e-mail: zhaoyun@suda.edu.cn\ \ ylcao@suda.edu.cn }
\date{\today}
 \maketitle

 \begin{center}
\begin{minipage}{120mm}
{\small {\bf Abstract.} This paper defines the pressure for
asymptotically  sub-additive potentials under a mistake function,
including the measure-theoretical  and the topological versions.
Using the  advanced techniques of ergodic theory and topological
dynamics, we reveals a variational principle for the new defined
topological pressure without any additional conditions on the
potentials and the compact metric space. }
\end{minipage}
\end{center}

\vskip0.5cm

\vskip0.5cm

 \section{Introduction}
  Throughout this paper,  $(X,T)$ denotes a topological
 dynamical systems(TDS for short) in the sense that $T:X\rightarrow X$ is a continuous transformation
 on the compact metric space $X$ with metric $d$.  The term $C(X)$ denotes the space of continuous functions from $X$ to
$\mathbb{R}$. Invariant Borel probability measures are are
associated with $(X,T)$. The terms $\mathcal{M}(X,T)$ and
$\mathcal{E}(X,T)$  represent  the space of $T-$invariant Borel
probability measures and the set of $T-$invariant ergodic Borel
probability measures, respectively.

 In classical ergodic theory, measure-theoretic entropy and topological
 entropy are important determinants  of complexity in dynamical
 systems. The important relationship between these two quantities
 is the well-known variational principle. Topological pressure is
 an
 important
 generalization of topological entropy. Ruelle first introduced the concept of
 topological pressure for additive potentials for expansive
 dynamical systems in \cite{rue}, in which he formulated a
 variational principle for topological pressure. Later, Walters \cite{wa1} generalized
these results to continuous maps on compact metric spaces. For an
arbitrary set, we emphasize that it need not be invariant or
compact, as it generalizes the notion of topological pressure
proposed by Pesin and Pitskel' in \cite{pp}, and these notions of
lower and upper capacity topological pressures introduced by Pesin
in \cite{pes}. The theories of topological pressure, variational
principle and equilibrium states play a fundamental role in
statistical mechanics, ergodic theory and dynamical systems, see
\cite{bo,rue2,k10}.

Since Bowen \cite{bowen}, topological pressure has become a basic
tool for studying dimension theory in conformal dynamical systems
\cite{pes2}.  To study dimension theory in non-conformal cases,
experts in dimension theory and dynamical systems introduced
 thermodynamic formalism  for non-additive potentials
\cite{ba,ba2,cfh,fal,fh,mu,zhang}.  Cao, Feng and Huang introduced
the sub-additive topological pressure via separated sets in
\cite{cfh} on general compact metric
 spaces, and obtained the variational principle for sub-additive potentials without any additional
 assumptions on the sub-additive potentials or the TDS
$(X,T)$.

This paper defines the pressure for asymptotically sub-additive
potentials under a mistake function, including the
measure-theoretical and the topological versions.  This paper also
obtains a variational principle for this newly defined topological
pressure. As a physical process evolves, it is natural for the
evolving process to change or produce some errors in the
evaluation of orbits.  However, a self-adaptable system should
decrease errors over time.  This is the motivation for this study
to investigate  the dynamical systems under a mistake function.
The following paragraphs provide some notations and definitions.

 For $x,y\in X$ and $n\in \mathbb{N}$,
$d_n(x,y):=\max\{d(T^i(x),T^i(y)):i=0,1,...,n-1\}$ gives a new
metric on $X$. The term  $B_n(x,\epsilon):=\{y\in
X:d_n(x,y)<\epsilon \}$ denotes a ball centered at $x$ with radius
$\epsilon$ under the metric $d_n$. Let $Z\subseteq X, n\in
\mathbb{N}$ and $\epsilon>0$. A set $F\subseteq Z$ is an
$(n,\epsilon)-$spanning set for $Z$ if for every $z\in Z$, there
exists $x\in F$ with $d_n(x,z)\leq \epsilon$. A set $E\subseteq Z$
is an $(n,\epsilon)-$separated set for $Z$ if for every $x,y\in E$
implies $d_n(x,y)>\epsilon$.  Given $\delta>0$ and $\mu\in
\mathcal{M}(X,T)$, a set $S$ is a $(n,\e,\delta)-$spanning set if
$\mu(\bigcup_{x\in S}B_n(x,\e))>1-\delta$.

 A sequence $\mathcal{F}=\{
f_n\}_{n=1}^{\infty}\subseteq C(X)$ is  an asymptotically
sub-additive potentials(ASP for short) on $X$, if for each $k
>0$, there exists a sub-additive potentials
$\Phi_k=\{\varphi_{n}^{k}\}_{n\geq 1}$, i.e.
$\varphi_{n+m}^{k}(x)\leq
\varphi_{n}^{k}(x)+\varphi_{m}^{k}(T^nx), \forall x\in X, n,m\in
\mathbb{N}$, such that
\[
\limsup_{n\rightarrow\infty}\frac{1}{n}||f_n-\varphi_{n}^{k}||\leq
\frac 1 k
\]
where $||f_n-\varphi_{n}^{k}||:=\max_{x\in
X}|f_n(x)-\varphi_{n}^{k}(x)|$. This kind of potential appears
naturally in the study of the dimension theory in dynamical
systems, see \cite{fh,zzc} for related examples.  Along with  Cao,
Feng and Huang's paper [6], Feng and Huang defined asymptotically
sub-additive topological pressure in \cite{fh} as follows:
\begin{eqnarray*}
P(T,\mathcal{F},n,\e)&=&\sup\{\sum_{y\in E}f_n(y): E \mbox{ is an
}(n,\epsilon)-\mbox{separated subset of } X\} \\
P(T,\mathcal{F})&=&\lim_{\epsilon\rightarrow
0}\limsup_{n\rightarrow \infty}\frac{1}{n}\log
P(T,\mathcal{F},n,\e)
   \end{eqnarray*}
the term $P(T,\mathcal{F})$ is the asymptotically sub-additive
topological pressure of $T$ with respect to(w.r.t.) $\mathcal{F}$.

Let $\mathcal{F}=\{ f_n\}_{n=1}^{\infty}$ be an ASP.   For a
$T$-invariant Borel probability measure $\mu$, let $h_{\mu}(T)$
denote the measure-theoretic entropy, and denote
$$\mathcal{F}_*(\mu)=\lim_{n\rightarrow \infty}\frac{1}{n}\int  f_n\,d\mu.$$
When $\mu\in \mathcal{E}(X,T)$,  the above limit exists
$\mu-$almost everywhere without integrating against $\mu$. See the
appendix in \cite{fh} for a proof of the above results. However,
it is easy to show that
$\mathcal{F}_*(\mu)=\lim\limits_{k\rightarrow\infty}\lim\limits_{n\rightarrow\infty}\frac
1 n\int \varphi_{n}^{k} \mathrm{d}\mu$.

With a minor modification of the proof in \cite{cfh}, Feng and
Huang obtained  the relationships among $P(T,\mathcal{F})$,
$h_{\mu}(T)$ and $\mathcal{F}_*(\mu)$ in \cite{fh}.

 \begin{theorem} \label{dl11} \rm
 \it Let $(X,T)$ be a TDS, and $\mathcal{F}=\{ f_n\}_{n\geq 1}$ an ASP. Then
$$
\begin{array}{l}
P(T,\mathcal{F})= \left \{
\begin{array}{l}
~~~~~~-\infty, ~~~~~~~~~~~~~~~~~~~~~\mbox{ if }\mathcal{F}_*(\mu)=-\infty \mbox{ for all } \mu\in \mathcal{M}(X,T),\\
\sup \{h_{\mu}(T)+\mathcal{F}_*(\mu):\mu\in \mathcal{M}(X,T),
\mathcal{F}_*(\mu)\neq -\infty\}, \mbox{ otherwise}.
\end{array}
\right.
\end{array}
$$
\end{theorem}

\begin{remark}\label{zhu1}
For each $\mu\in \mathcal{M}(X,T)$, let
$\mu=\int_{\mathcal{E}(X,T)} m\mathrm{\mathrm{d}}\tau(m)$ be its
ergodic decomposition. Thus, $h_{\mu}(T)=\int_{\mathcal{E}(X,T)}
h_m(T)\mathrm{\mathrm{d}}\tau(m)$ and
$\mathcal{F}_*(\mu)=\int_{\mathcal{E}(X,T)}\mathcal{F}_*(m)
\mathrm{\mathrm{d}}\tau(m)$, see \cite{k10} and \cite{fh} for
details. It is then possible to prove that
\begin{eqnarray*}
&&\sup \{h_{\mu}(T)+\mathcal{F}_*(\mu):\mu\in \mathcal{M}(X,T),
\mathcal{F}_*(\mu)\neq -\infty\}\\
&&=\sup \{h_{\mu}(T)+\mathcal{F}_*(\mu):\mu\in \mathcal{E}(X,T),
\mathcal{F}_*(\mu)\neq -\infty\}.
\end{eqnarray*}
Thus, we can replace $\mathcal{M}(X,T)$ with $\mathcal{E}(X,T)$ in
the supremum of theorem \ref{dl11}.
\end{remark}

The thermodynamic formalism  for a single function and  a sequence
of functions arose from various considerations in physics and
mathematics. This study extends thermodynamic formalism  to
asymptotically sub-additive potentials under a mistake function
without any condition on the potentials and the dynamics.

The remainder of this paper is organized as follows. Section 2
defines the pressure for ASP under a mistake function, including
the measure-theoretical  and the topological versions. And we
state our main result and  give some preliminary results. Section
3 provides the proof of the results. The  analysis in this study
relies on the techniques of ergodic theory and topological
dynamics.

 \section{Preliminaries}
This section first defines pressure for ASP under a mistake
function, and then presents the main results. The following
section presents  the proof.

First, recall the definitions of the  mistake function and mistake
dynamical  balls presented by Thompson \cite{th}.
\begin{definition}
Given $\e_0>0$ the function $g:\mathbb{N}\times
(0,\e_0]\rightarrow \mathbb{N}$ is called a mistake function if
for all $\e\in (0,\e_0]$ and all $n\in \mathbb{N}$, $g(n,\e)\leq
g(n+1,\e)$ and
\[
\lim_{n\rightarrow\infty}\frac{g(n,\e)}{n}=0.
\]
Given a mistake function $g$, if $\e>\e_0$  set
$g(n,\e)=g(n,\e_0)$.
\end{definition}

For any subset of integers $\Lambda\subset [0,N]$ we will use the
family of distances in the metric space $X$ given by
$d_\Lambda(x,y)=\max \{d(f^ix,f^iy):i\in\Lambda\}$ and consider
the balls $B_{\Lambda}(x,\e)=\{y\in X:d_\Lambda(x,y)<\e\}$.

\begin{definition}
Let $g$ be a mistake function and let $\e>0$ and $n\ge 1$. The
mistake dynamical ball $B_n(g;x,\e)$ of radius $\e$ and length $n$
associated to $g$ is defined as follows:
\begin{eqnarray*}
B_n(g;x,\e)&=&\{ y\in X\mid y\in B_{\Lambda}(x,\e)~\hbox{for
some}~\Lambda\in I(g;n,\e)\}\\
&=&\bigcup_{\Lambda\in I(g;n,\e)}B_{\Lambda}(x,\e)
\end{eqnarray*}
where $I(g;n,\e)=\{ \Lambda\subset [0,n-1]\cap\mathbb{N}\mid \#
\Lambda \geq n-g(n,\e)\}$ and $\# \Lambda$ denotes the cardinality
of the set $\Lambda$. A set $F\subset Z$ is  $(g;n,\e)-$separated
for $Z$ if for every $x,y\in F$ implies $d_\Lambda(x,y)>\e,\forall
\Lambda\in I(g;n,\e) $. The dual definition is as follows.  A set
$E\subset Z$ is $(g;n,\e)-$spanning for $Z$ if for all $z\in Z$,
there exists $x\in E$ and $\Lambda\in I(g;n,\e) $ such that
$d_\Lambda(x,y)\leq \e$. Given $\delta>0$ and $\mu\in
\mathcal{M}(X,T)$, a set $S$ is  $(g;n,\e,\delta)-$spanning set if
$\mu(\bigcup_{x\in S}B_n(g;x,\e))>1-\delta$.
\end{definition}

Let $\mathcal{F}=\{ f_n\}_{n=1}^{\infty}$ be an ASP and let $g$ be
a mistake function. For $\mu\in \mathcal{E}(X,T)$, the definition
of asymptotically sub-additive measure-theoretic pressure is as
follows:
\begin{eqnarray*}P_\mu(g;T,\mathcal{F},n,\e,\delta)&=&\inf \left\{\sum_{x\in
S}\exp[\sup_{y\in B_n(g;x,\e)}f_n(y)]\mid S\ \hbox{is a}\
(g;n,\e,\delta)- \hbox{spanning set}\right\}\\
P_\mu(g;T,\mathcal{F})&=&\lim_{\d\rightarrow 0}\lim_{\e\rightarrow
0}\liminf_{n\rightarrow\infty}\frac 1 n\log
P_\mu(g;T,\mathcal{F},n,\e,\delta)
\end{eqnarray*}
the term $P_\mu(g;T,\mathcal{F})$ is an asymptotically
sub-additive measure-theoretic pressure of $T$ w.r.t.
$\mathcal{F}$ under a mistake function $g$. The following theorem
presents the main findings of this paper, which imply that the
small errors cannot affect the important factors of dynamical
systems. The following section presents the  proof.

\begin{maintheorem}\label{thm.A}Let $(X,T)$ be a TDS, let $g$ be a mistake function, and let $\mathcal{F}=\{ f_n\}_{n=1}^{\infty}$ be an
ASP. For each $\mu\in \mathcal{E}(X,T)$ with
$\mathcal{F}_*(\mu)\neq -\infty$, we have
\begin{eqnarray*}
P_{\mu}(g;T,\mathcal{F})=\lim_{\e\rightarrow
0}\liminf_{n\rightarrow\infty}\frac 1 n\log
P_\mu(g;T,\mathcal{F},n,\e,\delta) =h_{\mu}(T)+\mathcal{F}_*(\mu).
\end{eqnarray*}The formula remains true if we replace the
$\liminf$ by $\limsup$, and the value taken by  the $\liminf$(or
$\limsup$) is independent of $\d$ and the mistake function $g$.
\end{maintheorem}

This result generalizes Katok's entropy formula \cite{katok}, and
the results in \cite{hlz} and \cite{zc}. The main virtue of this
approach  is that we do not require any condition on the ASP and
the TDS. The proof of the above theorem requires the following
lemma.

\begin{lemma} \label{yl21}Let $(X,T)$ be a TDS, let $g$ be a mistake function, and let $\mathcal{F}=\{ f_n\}_{n=1}^{\infty}$ be an
ASP. Given some $k>0$, there exist  sub-additive potentials
$\Phi_k=\{\varphi_{n}^{k}\}_{n\geq 1}$ such that for any positive
integer $l$ and small number $\eta>0$, there exists $\e_0>0$ so
that for any $0<\e<\e_0$, the following inequalities hold for
sufficiently large $n$
\[
\sup_{y\in B_n(g;x,\e)}f_n(y)\leq \sum_{i=0}^{n-1}\frac 1 l
\varphi_{l}^{k}(T^ix)+C( g(n,\e)+1)+n(\frac 1 k+\eta)
\]
where $C$ is a constant.
\end{lemma}
\begin{proof}
Given some $k>0$, since $\mathcal{F}=\{ f_n\}_{n=1}^{\infty}$ is
an ASP, there exist  sub-additive potentials
$\Phi_k=\{\varphi_{n}^{k}\}_{n\geq 1}$, such that $
\limsup_{n\rightarrow\infty}\frac{1}{n}||f_n-\varphi_{n}^{k}||\leq
\frac 1 k $. This implies that
\begin{eqnarray}\label{ds21}
f_n(x)\leq \varphi_{n}^{k}(x) +\frac n k,\ \ \forall x\in X
\end{eqnarray}
for sufficiently large $n$.

Let us fix any  positive integer $l$. Since $\frac 1 l
\varphi_{l}^{k}(x)$ is continuous,  for each $\eta>0$, there
exists $\e_0>0$ such that for any $0<\e<\e_0$, we have
\[
d(x,y)<\e\Rightarrow d(\frac 1 l \varphi_{l}^{k}(x),\frac 1 l
\varphi_{l}^{k}(y))<\eta.
\]
Note that for each $y\in B_n(g;x,\e)$,
 there exists $\Lambda\subset I(g;n,\e)$ so that $y\in B_\Lambda
 (x,\e)$, therefore
 \begin{eqnarray}\label{ds22}
\sum_{i=0}^{n-1}\frac 1 l
\varphi_{l}^{k}(T^iy)&\leq&\sum_{i\in\Lambda}(\frac 1 l
\varphi_{l}^{k}(T^ix)+\eta)+\sum_{i\notin\Lambda} ||\frac 1 l
\varphi_{l}^{k}||\notag \\
&\leq&\sum_{i=0}^{n-1}(\frac 1 l \varphi_{l}^{k}(T^ix)+\eta)+C_1
g(n,\e)
 \end{eqnarray}
where  $C_1=2(||\frac 1 l \varphi_{l}^{k}||+\eta)$.

For each sufficiently large $n$, it is possible to rewrite $n$ as
$n=sl+r$, where $0\leq s, 0\leq r<l$. Then, for any $0\leq j<l$,
we have
 \[
\varphi_{n}^{k}(x)\leq
\varphi_{j}^{k}(x)+\varphi_{l}^{k}(T^jx)+\cdots
+\varphi_{l}^{k}(T^{(s-2)l}T^jx)+\varphi_{l+r-j}^{k}(T^{(s-1)l}T^jx)
\]
where $\varphi_{0}^{k}(x)\equiv 0$. Summing $j$ from $0$ to $l-1$
leads to
\[
l\varphi_{n}^{k}(x)\leq 2l C_2+
\sum_{i=0}^{(s-1)l-1}\varphi_{l}^{k}(T^ix)
\]
where $C_2=\max_{j=1,\cdots 2l}\max_{x\in X}|\varphi_{j}^{k}(x)|$.
Hence,
\begin{eqnarray}\label{ds23}
\varphi_{n}^{k}(x)\leq 2 C_2+\sum_{i=0}^{(s-1)l-1}\frac 1
l\varphi_{l}^{k}(T^ix)\leq  4C_2+\sum_{i=0}^{n-1}\frac 1
l\varphi_{l}^{k}(T^ix).
\end{eqnarray}
Let $C=\max \{C_1,4C_2\}$, we have that
\begin{eqnarray*}
\sup_{y\in B_n(g;x,\e)}f_n(y)&\leq& \sup_{y\in B_n(g;x,\e)}
(C+\sum_{i=0}^{n-1}\frac 1 l\varphi_{l}^{k}(T^iy) +\frac n k)\\
&\leq&\sum_{i=0}^{n-1}(\frac 1 l \varphi_{l}^{k}(T^ix)+\eta)+C(
g(n,\e)+1)+\frac n k.
\end{eqnarray*}
where the  first inequality follows from (\ref{ds21}) and
(\ref{ds23}), and the second inequality follows from (\ref{ds22}).
This completes the proof of the lemma.
\end{proof}

Let $\mathcal{F}=\{ f_n\}_{n=1}^{\infty}$ be an ASP. The following
discussion defines the  topological version of asymptotically
sub-additive pressure under   a mistake function. This study first
gives an equivalent definition of asymptotically sub-additive
topological pressure via spanning set, and then gives a new
definition of  asymptotically sub-additive topological pressure
under a mistake function.

For each positive integer $n$ and $\e>0$, put
\begin{eqnarray*}
P^*(T,\mathcal{F},n,\e)&=&\inf\{\sum_{x\in F}\exp [\sup_{y\in
B_n(x,\e)}f_n(y)]: F \mbox{ is an }
   (n,\epsilon)-\mbox{spanning subset of } X\}\\
P^*(T,\mathcal{F})&=&\lim_{\epsilon\rightarrow
0}\limsup_{n\rightarrow \infty}\frac{1}{n}\log
P^*(T,\mathcal{F},n,\e).
\end{eqnarray*}
The following lemma says that this newly defined quantity equals
the asymptotically sub-additive topological pressure defined by
separated sets.

\begin{proposition} \label{mt21}
$P^*(T,\mathcal{F})=P(T,\mathcal{F})$.
\end{proposition}
\begin{proof}Let $F$ be an $(n,\epsilon/2)-$spanning subset of
$X$, and let  $E$ be  an $(n,\epsilon)-$separated subset of $X$.
Define a map $\phi:E\rightarrow F$ by choosing for each $x\in E$
some $\phi(x)\in F$ such that $d_n(x,\phi(x))\leq \e/2$. Then, it
is easy to see that $\phi$ is injective. Therefore,
\[
P^*(T,\mathcal{F},n,\e/2)\geq \sup\{\sum_{y\in E}e^{f_n(y)}: E
\mbox{ is an }
   (n,\epsilon)-\mbox{separated subset of } X\} .
\]
This immediately yields  $P^*(T,\mathcal{F})\geq
P(T,\mathcal{F})$.

Next, we prove that $P^*(T,\mathcal{F})\leq P(T,\mathcal{F})$.
Given $n\in \mathbb{N}$ and $\e>0$. Choose $x_1\in X$  with
$f_n(x_1)=\sup_{x\in X}f_n(x)$, and then choose $x_2\in X\setminus
B_n(x_1,\e)$ with $f_n(x_2)=\sup_{x\in X\setminus
B_n(x_1,\e)}f_n(x)$. We continue this process. More precise, in
step $m$ we choose $x_m\in X\setminus
\bigcup_{j=1}^{m-1}B_n(x_j,\e)$ with $f_n(x_m)=\sup_{x\in
X\setminus \bigcup_{j=1}^{m-1}B_n(x_j,\e)}f_n(x)$. This process
stops at some step $l$, and produces a maximal $(n,\e)-$separated
set $E=\{ x_1,x_2,\cdots, x_l\}$(meaning that $E$ is also an
$(n,\e)-$spanning set of $X$). Therefore,
\begin{eqnarray*}
P^*(T,\mathcal{F},n,\e)&\leq&\sum_{x\in E}\exp [\sup_{y\in
B_n(x,\e)}f_n(y)]=\sum_{x\in E}
e^{f_n(x)}\\
&\leq& \sup\{\sum_{y\in E}f_n(y): E \mbox{ is an }
   (n,\epsilon)-\mbox{separated subset of } X\}
\end{eqnarray*}
This  immediately implies  that $P^*(T,\mathcal{F})\leq
P(T,\mathcal{F})$, and completes the proof.
\end{proof}

Next, this study modifies the  definition of $P(T,\mathcal{F})$ to
define asymptotically sub-additive topological pressure under a
mistake function.

Let $\mathcal{F}=\{ f_n\}_{n=1}^{\infty}$ be an ASP and let $g$ be
a mistake function. For each $n\in \mathbb{N}$ and $\e>0$, put
\begin{eqnarray*}
P(g;T,\mathcal{F},n,\e)&=&\sup\{\sum_{x\in F}e^{f_n(x)}: F \mbox{
is an }
   (g;n,\epsilon)-\mbox{separated  subset of } X\}\\
P(g;T,\mathcal{F})&=&\lim_{\epsilon\rightarrow
0}\limsup_{n\rightarrow \infty}\frac{1}{n}\log
P(g;T,\mathcal{F},n,\e).
\end{eqnarray*}
The term $P(g;T,\mathcal{F})$ is the asymptotically sub-additive
topological pressure of $T$ w.r.t. $\mathcal{F}$  under a mistake
function $g$. The asymptotically sub-additive topological pressure
under mistake function  $P(g;T,\mathcal{F})$  equals
$P(T,\mathcal{F})$, which means that the dynamical system is self
adaptable if the amount of errors decrease as  time goes by.

\begin{maintheorem}\label{thm.B}Let $(X,T)$ be a TDS, let $g$ be a mistake function, and let $\mathcal{F}=\{ f_n\}_{n=1}^{\infty}$ be an
ASP. Then $P(g;T,\mathcal{F})=P(T,\mathcal{F})$.
\end{maintheorem}

Theorem B and theorem \ref{dl11} immediately imply the following
corollary, i.e.,  the variational principle for the asymptotically
sub-additive topological pressure under a mistake function.

\begin{corollary}Let $(X,T)$ be a TDS, let $g$ be a mistake function, and let $\mathcal{F}=\{ f_n\}_{n=1}^{\infty}$ be an
ASP. Then
$$
\begin{array}{l}
P(g;T,\mathcal{F})= \left \{
\begin{array}{l}
~~~~~~-\infty, ~~~~~~~~~~~~~~~~~~~~~\mbox{ if }\mathcal{F}_*(\mu)=-\infty \mbox{ for all } \mu\in \mathcal{M}(X,T),\\
\sup \{h_{\mu}(T)+\mathcal{F}_*(\mu):\mu\in \mathcal{M}(X,T),
\mathcal{F}_*(\mu)\neq -\infty\}, \mbox{ otherwise}.
\end{array}
\right.
\end{array}
$$
\end{corollary}

 To prove theorem B, we need an analogue  of
proposition  \ref{mt21}. Thus, we define
\begin{eqnarray*}
P^*(g;T,\mathcal{F},n,\e)&=&\inf\{\sum_{x\in F}\exp [\sup_{y\in
B_n(g;x,\e)}f_n(y)]: F \mbox{ is a }
   (g;n,\epsilon)-\mbox{spanning subset of } X\}\\
P^*(g;T,\mathcal{F})&=&\lim_{\epsilon\rightarrow
0}\limsup_{n\rightarrow \infty}\frac{1}{n}\log
P(g;T,\mathcal{F},n,\e).
\end{eqnarray*}
And the following lemma holds.

\begin{proposition}\label{mt22}$P(2g;T,\mathcal{F})\leq P^*(g;T,\mathcal{F})\leq P(g;T,\mathcal{F})$.
\end{proposition}
\begin{proof}
Let $F$ be an $(g;n,\epsilon/2)-$spanning subset of $X$, and let
$E$ be an $(2g;n,\epsilon)-$separated subset of $X$. Define a map
$\phi:E\rightarrow F$ by choosing for each $x\in E$ some
$\phi(x)\in F$ and some $\Lambda_x\in I(g;n,\e/2)$ such that
$d_{\Lambda_x}(x,\phi(x))\leq \e/2$. Suppose that $x,y\in E$ with
$x\neq y$, let $\Lambda=\Lambda_x\cap\Lambda_y$. Since $\Lambda\in
I(2g;n,\e/2)$,  $d_\Lambda(\phi(x),\phi(y))>0$ and thus
$\phi(x)\neq\phi(y)$. Hence, $\phi$ is injective. Therefore,
\[
P^*(g;T,\mathcal{F},n,\e/2)\geq \sup\{\sum_{y\in E}e^{f_n(y)}: E
\mbox{ is an }
   (2g;n,\epsilon)-\mbox{separated subset of } X\} .
\]
This immediately shows that $P^*(g;T,\mathcal{F})\geq
P(2g;T,\mathcal{F})$.

Next, we prove that $P^*(g;T,\mathcal{F})\leq P(g;T,\mathcal{F})$.
Given $n\in \mathbb{N}$ and $\e>0$, choose $x_1\in X$  with
$f_n(x_1)=\sup_{x\in X}f_n(x)$, and then choose $x_2\in X\setminus
B_n(g;x_1,\e)$ with $f_n(x_2)=\sup_{x\in X\setminus
B_n(g;x_1,\e)}f_n(x)$. We continue this process. More precise, in
step $m$  choose $x_m\in X\setminus
\bigcup_{j=1}^{m-1}B_n(g;x_j,\e)$ with $f_n(x_m)=\sup_{x\in
X\setminus \bigcup_{j=1}^{m-1}B_n(g;x_j,\e)}f_n(x)$. This process
 stops at some step $l$, producing  a maximal
$(g;n,\e)-$separated set $E=\{ x_1,x_2,\cdots, x_l\}$(meaning that
$E$ is also an $(g;n,\e)-$spanning set of $X$, see lemma 3.3 in
\cite{th} for a proof). Therefore,
\begin{eqnarray*}
P^*(g;T,\mathcal{F},n,\e)&\leq&\sum_{x\in E}\exp [\sup_{y\in
B_n(g;x,\e)}f_n(y)]=\sum_{x\in E}
e^{f_n(x)}\\
&\leq& \sup\{\sum_{y\in E}f_n(y): E \mbox{ is an }
   (g;n,\epsilon)-\mbox{separated subset of } X\}
\end{eqnarray*}
This  immediately implies  that $P^*(g;T,\mathcal{F})\leq
P(g;T,\mathcal{F})$, and completes the proof of the lemma.
\end{proof}

\section{Proof of main results}
This section proves theorems A and B presented  in the former
section.

\subsection{ Proof of Theorem A}This subsection gives the proof of theorem A
by following the arguments in \cite{pes2} and \cite{th}, but the
proof here is more complicated. This means  that the
asymptotically sub-additive measure-theoretic  pressure is stable
under a mistake function.
\begin{proof}Assume that $\mu\in
\mathcal{E}(X,T)$ with $\mathcal{F}_*(\mu)\neq -\infty$. Note that
$B_n(x,\e)\subset B_n(g;x,\e)$ implies that an
$(n,\e,\d)-$spanning set must be a $(g;n,\e,\d)-$spanning set, and
this leads to the following inequality
\begin{eqnarray*}
&&P_\mu(g;T,\mathcal{F},n,\e,\delta)\leq\inf \left\{\sum_{x\in
S}\exp[\sup_{y\in B_n(g;x,\e)}f_n(y)]\mid S\ \hbox{is a}\
(n,\e,\delta)- \hbox{spanning set}\right\}\\
&&\leq e^{[C( g(n,\e)+1)+n(\frac 1 k+\eta)]}\inf \left\{\sum_{x\in
S}\exp\sum_{i=0}^{n-1}\frac 1 l \varphi_{l}^{k}(T^ix)\mid S\
\hbox{is a}\ (n,\e,\delta)- \hbox{spanning set}\right\}
\end{eqnarray*}
where the second inequality follows from lemma \ref{yl21}. The
terms $l,C,\eta,k$ and $\Phi_k=\{\varphi_{n}^{k}\}_{n\geq 1}$ are
all the same as lemma \ref{yl21}. Previous authors \cite{hlz}
proved that
\begin{eqnarray*}
&&\lim_{\e\rightarrow 0}\liminf_{n\rightarrow\infty}\frac 1 n\log\
 \inf \left\{\sum_{x\in S}\exp\sum_{i=0}^{n-1}\frac 1 l
\varphi_{l}^{k}(T^ix)\mid S\ \hbox{is a}\ (n,\e,\delta)-
\hbox{spanning set}\right\}\\
&&=h_{\mu}(T)+\int \frac 1 l \varphi_{l}^{k}(x)\mathrm{d}\mu.
\end{eqnarray*}
Therefore, based on the fact that $g$ is a mistake function,
\begin{eqnarray*}
P_\mu(g;T,\mathcal{F})\leq h_{\mu}(T)+\int \frac 1 l
\varphi_{l}^{k}(x)\mathrm{d}\mu+\frac{1}{k}+\eta.
\end{eqnarray*}
Let $l\rightarrow\infty$ and $k\rightarrow\infty$, and the
arbitrariness of $\eta$ implies that $P_\mu(g;T,\mathcal{F})\leq
h_{\mu}(T)+\mathcal{F}_*(\mu)$.

 Now, we turn to prove the reverse inequality that $P_\mu(g;T,\mathcal{F})\geq
 h_{\mu}(T)+\mathcal{F}_*(\mu)$.  This method is similar to the proof of theorem A2.1 in \cite{pes2}. For each $\eta >0$, there exists
 $0<\gamma\leq \eta$, a finite partition $\xi=\{C_1,C_2,\cdots ,
 C_m\}$  and a finite open cover $\mathcal{U}=\{ U_1,U_2,\cdots ,
 U_k\}$ of $X$, where $k\geq m$. The following properties
 holds(using the regularity of the measure $\mu$):

 (1)$|U_i|\leq \eta$ and $|C_j|\leq\eta$, $1\leq i\leq k, 1\leq j\leq
 m$, here $|\cdot|$ denote the diameter of  set;

 (2)$\overline{U_i}\subset C_i$, $1\leq i\leq
 m$, where $\overline{A}$ denotes the closure of the set $A$;

 (3)$\mu (C_i\setminus U_i)\leq \gamma$, $1\leq i\leq
 m$ and $\mu(\bigcup_{i=m+1}^{k}U_i)\leq \g$;

 (4) $2\g \log m\leq \eta$.

 Next,  fix $\eta$  so $1-\d>\eta>0$ and take the corresponding $\g$,
 partition $\xi$ and covering $\mathcal{U}$. Fix $Z\subset X$ with
 $\mu(Z)>1-\d$ and put $t_n(x):=\sharp \{ 0\leq l<n: T^lx\in
 \bigcup_{i=m+1}^{k}U_i\}$. Let
 $\xi_n=\bigvee_{i=0}^{n-1}T^{-i}\xi$ and $\xi_n(x)$ denote the
 element of $\xi_n$ contains $x$.

 We claim that: there exists $A\subset
 Z$ and $N>0$ with $\mu(A)\geq \mu(Z)-\g$ such that for every $x\in
 A$ and $n\geq N$, we have (i) $t_n(x)\leq 2\g n$; (ii) $\mu(\xi_n(x))\leq \exp[-(h_{\mu}(T,\xi)-\g)n]
 $; (iii) $\mathcal{F}_*(\mu)-\g\leq\frac{1}{n}f_n(x)\leq
 \mathcal{F}_*(\mu)+\g$.

 {\it Proof of the claim:} Let $g=\chi_{\bigcup_{i=m+1}^{k}U_i}$,
 then $t_n(x)=\sum_{j=0}^{n-1}g(T^ix)$. According to the Birkhoff ergodic
 theorem and  Egorov theorem, we can find a set $A_1\subset Z$
 with $\mu(A_1)\geq \mu(Z)-\frac \g 3$ such that
 \[
 \lim_{n\rightarrow\infty}\frac 1 n t_n(x)=\lim_{n\rightarrow\infty}\frac 1
 n \sum_{j=0}^{n-1}g(T^ix) = \int g \mathrm{d}\mu
 =\mu(\bigcup_{i=m+1}^{k}U_i)\leq \g
 \]
 holds uniformly on $A_1$. Therefore, we can choose $N_1$ such that if
 $n\geq N_1$ and $x\in A_1$, then $t_n(x)\leq 2\g n$. Using the
 Shannon -Mcmillan-Brieman  thereom and Egorov theorem, it is possible to  find a set $A_2\subset Z$
 with $\mu(A_2)\geq \mu(Z)-\frac \g 3$. By the same argument, it is possible to choose $N_2$ such that if
 $n\geq N_2$ and $x\in A_2$, then $\mu(\xi_n(x))\leq \exp[-(h_{\mu}(T,\xi)-\g)n]
 $. Then,  using Egorov theorem and the  fact that
 \[
 \lim_{n\rightarrow}\frac 1 n f_n(x)=\mathcal{F}_*(\mu)(\neq -\infty),\ \
 \mu-a.e.\ x\in X.
 \]we can find a set $A_3\subset Z$
 with $\mu(A_3)\geq \mu(Z)-\frac \g 3$. By the same argument, it is possible to  choose $N_3$ such that if
 $n\geq N_3$ and $x\in A_3$, then  $\mathcal{F}_*(\mu)-\g\leq\frac{1}{n}f_n(x)\leq
 \mathcal{F}_*(\mu)+\g$. Set $A=A_1\cap A_2\cap A_3$ and
 $n=\max\{N_1,N_2,N_3\}$ to prove the claim.

 Set $\xi_{n}^{*}:=\{ \xi_n(x)\in \xi_n \mid  \xi_n(x)\cap A\neq \emptyset
 \}$. Using (ii) of the claim shows that
 \begin{eqnarray} \label{ds34}
 \sharp \xi_{n}^{*} \geq \sum_{\xi_n(x)\in \xi_{n}^{*}}\mu
 (\xi_n(x))\exp[(h_{\mu}(T,\xi)-\g)n]\geq \mu(A)
 \exp[(h_{\mu}(T,\xi)-\g)n],\ \forall n\geq N
 \end{eqnarray}
Let $2\e$ be the Lebesgue number of the open cover $\mathcal{U}$
and let $S$ be a $(g;n,\e)-$spanning set for $Z$. Picking a
suitable $\Lambda_x\in I(g;n,\e)$ leads to $Z\subset \bigcup_{x\in
S}\overline{B}_{\Lambda_x}(x,\e)$. Let $S'\subset S$ such that
$\overline{B}_{\Lambda_x}(x,\e)\cap A\neq \emptyset$ for each
$x\in S'$. Fix $x\in S'$ and $B=\overline{B}_{\Lambda_x}(x,\e)$,
let $\xi_{\Lambda_x}:=\bigvee_{j\in \Lambda_x}T^{-j}\xi$,
$p(B,\xi_{\Lambda_x}):=\sharp \{C\in  \xi_{\Lambda_x}\mid C\cap
A\cap B\neq\emptyset \}$ and $p(B,\xi_n):=\sharp \{C\in \xi_n\mid
C\cap A\cap B\neq\emptyset \}$.

We now estimate the number $p(B,\xi_{\Lambda_x})$. Note that
$\overline{B}(T^jx,\e)\subset U_{i_l}$ for some $U_{i_l}\in
\mathcal{U}$, since $2\e$ is the Lebesgue number of the open cover
$\mathcal{U}$. If $i_l\in \{ 1,2,\cdots , m\}$ then
$T^{-l}U_{i_l}\subset T^{-l}C_{i_l}$. If $i_l\in \{ m+1,\cdots ,
k\}$, then there are at most $m$ sets of the form $T^{-l}C_{i_l}$
may have non-empty intersection with $T^{-l}U_{i_l}$. Using (i) of
the claim shows that
\[
p(B,\xi_{\Lambda_x})\leq m^{2\g n}=\exp(2\g n\log m).
\]
Therefore,
\[
p(B,\xi_n)\leq p(B,\xi_{\Lambda_x}) m^{g(n,\e)}\leq \exp [(2\g
n+g(n,\e))\log m].
\]
It follows that
\begin{eqnarray} \label{ds35}
\sharp \xi_{n}^{*}\leq \sum_{x\in
S'}p(\overline{B}_{\Lambda_x}(x,\e),\xi_n)\leq \sharp S'  \exp
[(2\g n+g(n,\e))\log m].
\end{eqnarray}
Therefore,
\begin{eqnarray*}
\sum_{x\in S}\exp[\sup_{y\in B_n(g;x,\e)}f_n(y)]&\geq& \sum_{x\in
S'}\exp[\sup_{y\in B_n(g;x,\e)}f_n(y)] \geq\sharp S'
\exp[n(\mathcal{F}_*(\mu)-\g)]\\
&\geq&\mu(A)\exp
[(h_{\mu}(T,\xi)+\mathcal{F}_*(\mu)-2\g)n-(g(n,\e)+2n\g )\log m]
\end{eqnarray*}
where the second inequality follows from the fact that
$B_n(g;x,\e)\cap A\neq \emptyset$ for each $x\in S'$ and (iii) of
the claim, and the third inequality follows from (\ref{ds34}) and
(\ref{ds35}). This leads to
\[
\frac 1 n \log P_{\mu}(g;T,\mathcal{F},n,\e,\d)\geq  \frac 1 n
\log \mu(A)
+h_{\mu}(T,\xi)+\mathcal{F}_*(\mu)-2\g-\frac{(g(n,\e)+2n\g)\log
m}{n}
\]
Since $\g<\eta$, $2\g \log m<\eta$, $\frac{g(n,\e)}{n}\rightarrow
0$ as $n\rightarrow\infty$,  $|\xi|:=\max_{1\leq i\leq
m}{|C_i|}<\eta$, and $\eta$ is arbitrary,
\[
P_{\mu}(g;T,\mathcal{F})\geq h_{\mu}(T)+\mathcal{F}_*(\mu).
\]
This completes the proof of the theorem.\end{proof}

\subsection{\bf Proof of Theorem B}This subsection combines the results in theorem A and proposition
\ref{mt22} to give the proof of theorem B. This proof says that
the asymptotically sub-additive  topological pressure is stable
under a mistake function.
\begin{proof} If $E$ is a
$(g;n,\e)-$separated set, then $E$ must be an $(n,\e)-$separated
set. Therefore,
\[
P(g;T,\mathcal{F},n,\e)\leq \sup\left\{\sum_{y\in E}f_n(y): E
\mbox{ is an }
   (n,\epsilon)-\mbox{separated subset of } X\right\}.
\]
Hence, $P(g;T,\mathcal{F})\leq P(T,\mathcal{F})$.

Now it is enough to prove that $P(g;T,\mathcal{F})\geq \sup
\{h_{\mu}(T)+\mathcal{F}_*(\mu):\mu\in \mathcal{E}(X,T),
\mathcal{F}_*(\mu)\neq -\infty\} $ by remark \ref{zhu1}. To
illustrate this statement, for each $\mu\in \mathcal{E}(X,T)$ with
$\mathcal{F}_*(\mu)\neq -\infty$, a $(g;n,\e)-$spanning set must a
$(g;n,\e,\d)-$spanning set. Therefore,
\[
P^*(g;T,\mathcal{F},n,\e)\geq P_{\mu}(g;T,\mathcal{F},n,\e,\d).
\]
According  to theorem A and proposition \ref{mt22},
\[
P(g;T,\mathcal{F})\geq P^*(g;T,\mathcal{F})\geq
h_{\mu}(T)+\mathcal{F}_*(\mu),\ \forall \mu\in \mathcal{E}(X,T)\
\hbox{with}\  \mathcal{F}_*(\mu)\neq -\infty.
\]
Combining the above arguments, theorem B immediately
follows.\end{proof}

 \noindent {\bf Acknowledgements.} Part of this work was carried out when Cheng and Zhao visited NCTS,  authors
sincerely appreciates the warm hospitality of the host. Cheng is
partially supported by NSC Grants 99-2115-M-034-001, Zhao is
partially supported by NSFC(11001191), NSF in Jiangsu
province(09KJB110007) and a Pre-research Project of Suzhou
University, Cao is partially supported by NSFC(10971151)  and the
973 Project (2007CB814800).

\end{document}